\documentclass[10pt]{amsart}
\usepackage{amssymb,amsthm,amsmath,amsfonts}
\usepackage{hyperref}
\usepackage{enumerate}

\makeatletter
\g@addto@macro\th@plain{\thm@headpunct{}}
\makeatother

\newtheorem{theorem}{Theorem}

\newtheorem{remark}{Remark}

\newcommand{\xx}{ {\textbf x} }
\newcommand{\ab}{ {\textbf a} }
\newcommand{\bb}{ {\textbf b} }
\newcommand{\cb}{ {\textbf c} }

\newcommand{\yy}{ {\textbf y} }

\newcommand{\zz}{ {\textbf z} }
\newcommand{\ee}{ {\textbf e} }
\newcommand{\ub}{ {\textbf u} }
\newcommand{\vb}{ {\textbf v} }

\newcommand{\VV}{ \Omega }
\newcommand{\RR}{\mathbb{R}}
\newcommand{\KK}{\mathbb{K}}

\newcommand{\LL}{\mathbb{L}}
\newcommand{\PP}{\mathbb{P}}
\newcommand{\En}{\mathbb{E}}

\newcommand{\DD}{\mathcal{D}}

\newcommand{\TT}{\mathcal{T}}

\newcommand{\Trace}{\mathrm{Trace}\,}
\newcommand{\DDet}{\mathrm{Det}}

\providecommand{\norm}[1]{\lVert#1\rVert}

\providecommand{\scalar}[1]{\left\langle#1\right\rangle}

\newcommand{\gw}{ {\mathfrak g} }
\newcommand{\ww}{ {\mathfrak w} }
\newcommand{\we}{ {\mathfrak w}_\ee }

\title[Characterization of beta distribution]{Characterization of beta distribution on symmetric cones}
\author[B. Ko\l{}odziejek]{Bartosz Ko\l{}odziejek}
\address{Faculty of Mathematics and Information Science\\Warsaw University of Technology\\Koszykowa 75\\00-661 Warsaw, Poland}
\email{kolodziejekb@mini.pw.edu.pl}


\begin{document}

\begin{abstract}
In the paper we generalize the following characterization of beta distribution to the symmetric cone setting: let $X$ and $Y$ be independent, non-degenerate random variables with values in $(0,1)$, then $U=1-XY$ and $V=\frac{1-X}{U}$ are independent if and only if there exist positive numbers $p_i$, $i=1,2,3$, such that $X$ and $Y$ follow beta distributions with parameters $(p_1+p_3,p_2)$ and $(p_3,p_1)$, respectively.
\end{abstract}
\maketitle

\section{Introduction}
In the paper we generalize the following characterization of beta distribution to random matrices and, more generally, to random variables valued in the symmetric cone: \emph{let $X$ and $Y$ be independent, non-degenerate random variables with values in $(0,1)$, then $U=1-XY$ and $V=\frac{1-X}{U}$ are independent if and only if there exist positive numbers $p_i$, $i=1,2,3$, such that $X$ and $Y$ follow beta distributions with parameters $(p_1+p_3,p_2)$ and $(p_3,p_1)$, respectively}. This univariate result was proved in \cite{WesBeta} under additional assumptions that $X$ and $Y$ have densities, which are strictly positive on $(0,1)$ and are log-locally integrable.
Regularity assumption on densities was removed in the work of \cite{LajMe2009}. It turns out that the existence of densities assumption is redundant, what was shown in \cite{WesBeta2}.

Here we are interested in a generalization of density versions of the beta characterization, when random variables are valued in the cone $\VV_+$ of $r\times r$ positive definite symmetric real matrices.
Define the analogue of $(0,1)$ interval in $\VV_+$: $\DD_+=\left\{\xx\in\VV_+\colon I-\xx\in\VV_+ \right\}$, where $I$ is the identity matrix. 
Beta distribution on symmetric cone $\VV_+$ with parameters $(p,q)$ for $p,q>\dim\VV_+/r-1$ is defined by its density
\begin{align*}
\mathcal{B}(p,q)(\mathrm{d}\xx)=\frac{1}{\mathcal{B}_{\VV_+}(p,q)}(\det\xx)^{p-\dim\VV/r}\det(I-\xx)^{q-\dim\VV_+/r}I_{\DD_+}(\xx)\,\mathrm{d}\xx,\quad\xx\in\VV_+,
\end{align*}
where $\mathcal{B}_{\VV_+(p,q)}$ is the normalizing constant. For any $\xx\in\VV_+$ there exists unique $\yy\in\VV_+$ such that $\yy^2=\xx$. Matrix $\yy$ is denoted by $\yy=\xx^{1/2}$.
We will show that if $X$ and $Y$ are independent random variables valued in $\DD_+$, having continuous densities, which are strictly positive on $\DD_+$, then \mbox{$U=I-X^{1/2}\cdot Y\cdot X^{1/2}$} and \mbox{$V=U^{-1/2}\cdot(I-X)\cdot U^{-1/2}$} are independent if and only if there exist numbers $p_i>\dim\VV_+/r-1$, $i=1,2,3$, such that $X$ and $Y$ follow matrix-variate beta distribution with parameters $(p_1+p_3,p_2)$ and $(p_3,p_1)$, respectively. 

Actually, we will consider much more general form of transformation of random variables, which is defined through, so-called, multiplication algorithm. A multiplication algorithm is a mapping $w\colon\VV_+\mapsto GL(r,\RR)$ such that $w(\xx)\cdot w^\top(\xx)=\xx$ for any $\xx\in\VV_+$, where $GL(r,\RR)$ is the group of invertible $r\times r$ matrices and $w^\top(\xx)$ is the transpose of $w(\xx)$. Multiplication algorithms (actually their inverses called division algorithms) were introduced by \cite{OlRu1962} alongside the characterization of Wishart probability distribution (see also \cite{CaLe1996} for generalization to symmetric cone setting). The two basic examples of multiplication algorithms are $w_1(\xx)=\xx^{1/2}$ ($\xx^{1/2}$ being the unique positive definite symmetric square root of $\xx$) and $w_2(\xx)=t_\xx$, where $t_\xx$ is the lower triangular matrix from the Cholesky decomposition of $\xx=t_\xx\cdot t_\xx^\top$. 

We will consider the independence of $U=I-w(X)\cdot Y \cdot w^\top(X)$ and $V=(\widetilde{w}(U))^{-1}\cdot (I-X)\cdot (\widetilde{w}^\top(U))^{-1}$, where $w$ and $\widetilde{w}$ are two multiplication algorithms satisfying additionally some natural conditions. It turns out that, depending on the choice of multiplication algorithms, the characterized distribution may not be the beta distribution (see Theorem \ref{betathS}). For example, when $w=\widetilde{w}=w_2$ the condition of independence of $U$ and $V$ characterizes wider family of distributions called beta-Riesz, which include beta distribution as a special case. 

As in the famous Lukacs-Olkin-Rubin Theorem (see \cite{OlRu1964} for $\VV_+$ case and \cite{CaLe1996} for all symmetric cones) the assumption of invariance under the group of automorphisms of distributions of $X$ and $Y$ is considered. The distribution of $X$ is said to be invariant under the group of automorphisms if $O\cdot X\cdot O^\top\stackrel{d}{=}X$ for any orthogonal matrix $O$. This approach leads to a characterization of beta distribution regardless of the choice of multiplication algorithms (see Theorem \ref{betathK}).

We cannot give the explicit formula for densities for any multiplication algorithms. In general case, the densities are given in terms of, so-called, $w$-logarithmic Cauchy functions, that is, functions that satisfy the following functional equation 
\begin{align*}
f(\xx)+f(w(I)\cdot \yy\cdot w^\top(I))=f(w(\xx)\cdot\yy\cdot w^\top(\xx)),\quad (\xx,\yy)\in\VV_+.
\end{align*}
The form of $w$-logarithmic Cauchy functions without any regularity assumptions for two basic examples of multiplication algorithms were recently considered in \cite{wC2013}. Later on we will write $\ww(\xx)$ for the the linear operator acting on $\VV_+$ such that $\ww(\xx)\yy=w(\xx)\cdot \yy\cdot w^\top(\xx)$. $\ww(\xx)$ will also be termed a multiplication algorithm.

Analogous characterization of Wishart distribution, when densities of respective random variables are given in terms of $w$-logarithmic functions is given in \cite{BKLOR2014}. Unfortunately, we cannot answer the question whether there exists multiplication algorithm resulting in characterizing other distribution than beta or beta-Riesz. Moreover, the removal of the assumption of the existence of densities remains a challenge.
 
The idea of the proof is analogous to that of \cite{WesBeta}. The independence condition gives us the functional equation for densities, which is then solved.  
As was observed in \cite{LajMe2009}, in univariate case, the independence condition leads to the generalized fundamental equation of information, that is
\begin{align*}
F(x)+G\left(\tfrac{y}{1-x}\right)=H(y)+K\left(\tfrac{x}{1-y}\right),
\end{align*}
where $(x,y)\in \mathrm{D}_0=\left\{(x,y)\in(0,1)^2\colon x+y\in(0,1)\right\}$ and $F,G,H,K\colon(0,1)\to\RR$ are unknown functions.
Our proof will heavily rely on the solution to the generalization of this equation to the cone $\VV_+$, which was given in \cite{BKInf2014}.

Similar characterization of beta distribution for random matrices was proved under numerous additional assumptions in \cite{HaRe2009}. The characterization of $2\times2$ matrix-variate beta distribution was also given by \cite{BWBeta}, but the characterization condition was of a different nature.

All above considerations can be generalized to the symmetric cones, of which $\VV_+$ is the prime example. The paper is organized as follows. In the next section we give necessary introduction to the theory of symmetric cones. Next, in Section 3 we define beta and beta-Riesz probability distributions on symmetric cones. Main theorems are stated and proved in Section $4$. Section $5$ is devoted to the analysis of the problem, when $X$ and $Y$ have distributions invariant under the group of automorphisms.

\section{Preliminaries}
In this section we recall basic facts of the theory of symmetric cones, which are needed in the paper. For further details we refer to \cite{FaKo1994}. 

A \textit{Euclidean Jordan algebra} is a Euclidean space $\En$ (endowed with scalar product denoted $\scalar{\xx,\yy}$) equipped with a bilinear mapping (product)
\begin{align*}
\En\times\En \ni \left(\xx,\yy\right)\mapsto \xx\yy\in\En
\end{align*}
and a neutral element $\ee$ in $\En$ such that for all $\xx$, $\yy$, $\zz$ in $\En$:
\begin{enumerate}[(i)]
	\item $\xx\yy=\yy\xx$, 
	\item $\xx(\xx^2\yy)=\xx^2(\xx\yy)$,
	\item $\xx\ee=\xx$,
	\item $\scalar{\xx,\yy\zz}=\scalar{\xx\yy,\zz}$.
\end{enumerate}
For $\xx\in\En$ let $\LL(\xx)\colon \En\to\En$ be linear map defined by
\begin{align*}
\LL(\xx)\yy=\xx\yy,
\end{align*}
and define 
\begin{align*}
\PP(\xx)=2\LL^2(\xx)-\LL\left(\xx^2\right).
\end{align*} 
Let $\mathrm{End}(\En)$ denote the space of endomorphisms of $\En$. The map $\PP\colon \En\mapsto \mathrm{End}(\En)$ is called the \emph{quadratic representation} of $\En$.

An element $\xx$ is said to be \emph{invertible} if there exists an element $\yy$ in $\En$ such that $\LL(\xx)\yy=\ee$. Then $\yy$ is called the \emph{inverse of} $\xx$ and is denoted by $\yy=\xx^{-1}$. Note that the inverse of $\xx$ is unique. It can be shown that $\xx$ is invertible if and only if $\PP(\xx)$ is invertible and in this case $\left(\PP(\xx)\right)^{-1} =\PP\left(\xx^{-1}\right)$.

A Euclidean Jordan algebra $\En$ is said to be \emph{simple} if it is not a \mbox{Cartesian} product of two Euclidean Jordan algebras of positive dimensions. Up to linear isomorphism there are only five kinds of Euclidean simple Jordan algebras. Let $\mathbb{K}$ denote either the real numbers $\RR$, the complex ones $\mathbb{C}$, quaternions $\mathbb{H}$ or the octonions $\mathbb{O}$, and write $S_r(\mathbb{K})$ for the space of $r\times r$ Hermitian matrices valued in $\mathbb{K}$, endowed with the Euclidean structure $\scalar{\xx,\yy}=\Trace(\xx\cdot\bar{\yy})$ and with the Jordan product
\begin{align}\label{defL}
\xx\yy=\tfrac{1}{2}(\xx\cdot\yy+\yy\cdot\xx),
\end{align}
where $\xx\cdot\yy$ denotes the ordinary product of matrices and $\bar{\yy}$ is the conjugate of $\yy$. Then $S_r(\RR)$, $r\geq 1$, $S_r(\mathbb{C})$, $r\geq 2$, $S_r(\mathbb{H})$, $r\geq 2$, and the exceptional $S_3(\mathbb{O})$ are the first four kinds of Euclidean simple Jordan algebras.  Note that in this case if $\KK\neq \mathbb{O}$, then
\begin{align}\label{defP}
\PP(\yy)\xx=\yy\cdot\xx\cdot\yy.
\end{align}
The fifth kind is the Euclidean space $\RR^{n+1}$, $n\geq 2$, with Jordan product
\begin{align}\label{scL}\begin{split}
\left(x_0,\ldots, x_n\right)\left(y_0,\ldots,y_n\right) =\left(\sum_{i=0}^n x_i y_i,x_0y_1+y_0x_1,\ldots,x_0y_n+y_0x_n\right).
\end{split}
\end{align}

To each Euclidean simple Jordan algebra one can attach the set of Jordan squares
\begin{align*}
\bar{\VV}=\left\{\xx^2\colon\xx\in\En \right\}.
\end{align*}
The interior $\VV$ of $\bar{\VV}$ is a symmetric cone.
Moreover $\VV$ is \emph{irreducible}, i.e. it is not the Cartesian product of two convex cones. One can prove that an open convex cone is symmetric and irreducible if and only if it is the cone $\VV$ of some Euclidean simple Jordan algebra. Each simple Jordan algebra corresponds to a symmetric cone, hence there exist up to linear isomorphism also only five kinds of symmetric cones. The cone corresponding to the Euclidean Jordan algebra $\RR^{n+1}$ equipped with Jordan product \eqref{scL} is called the Lorentz cone. 

We denote by $G(\En)$ the subgroup of the linear group $GL(\En)$ of linear automorphisms which preserves $\VV$, and we denote by $G$ the connected component of $G(\En)$ containing the identity.  Recall that if $\En=S_r(\RR)$ and $GL(r,\RR)$ is the group of invertible $r\times r$ matrices, elements of $G(\En)$ are the maps $g\colon\En\to\En$ such that there exists $\ab\in GL(r,\RR)$ with
\begin{align*}
g(\xx)=\ab\cdot\xx\cdot\ab^\top.
\end{align*}
We define $K=G\cap O(\En)$, where $O(\En)$ is the orthogonal group of $\En$. It can be shown that 
\begin{align}\label{defk}
K=\{ k\in G\colon k\ee=\ee \}.
\end{align} 

A \emph{multiplication algorithm} is a map $\VV\to G\colon \xx\mapsto \ww(\xx)$ such that $\ww(\xx)\ee=\xx$ for all $\xx\in\VV$. This concept is consistent with, so-called, division algorithm $\gw$, which was introduced by \cite{OlRu1962} and \cite{CaLe1996}, that is a mapping $\VV\ni\xx\mapsto \gw(\xx)\in G$ such that $\gw(\xx)\xx=\ee$ for any $\xx\in\VV$. If $\ww$ is a multiplication algorithm then $\gw=\ww^{-1}$ is a division algorithm and vice versa, if $\gw$ is a division algorithm then $\ww=\gw^{-1}$ is a multiplication algorithm. 

By \cite[Proposition III.4.3]{FaKo1994}, for any $g$ in the group $G$,
\begin{align*}
\det(g\xx)=(\DDet\,g)^{r/\dim\VV}\det\xx,
\end{align*}
where $\DDet$ denotes the determinant in the space of endomorphisms on $\VV$. Inserting a multiplication algorithm $g=\ww(\yy)$, $\yy\in\VV$, and $\xx=\ee$ we obtain
\begin{align}\label{zzz}
\DDet\left(\ww(\yy)\right) =(\det\yy)^{\dim\VV/r}
\end{align}
and hence
\begin{align}\label{detw}
\det(\ww(\yy)\xx) =\det\yy\det\xx
\end{align}
for any $\xx,\yy\in\VV$.

One of two important examples of multiplication algorithms is the map $\ww_1(\xx)=\PP\left(\xx^{1/2}\right)$. The remaining part of this section is to give the necessary background for the definition of the second basic example of multiplication algorithm, the one connected with Cholesky decomposition.

We will now introduce a very useful decomposition in $\En$, called \emph{spectral decomposition}. An element $\cb\in\En$ is said to be a \emph{idempotent} if $\cb\cb=\cb\neq 0$. Idempotents $\ab$ and $\bb$ are \emph{orthogonal} if $\ab\bb=0$. Idempotent $\cb$ is \emph{primitive} if $\cb$ is not a sum of two non-null idempotents. A \emph{complete system of primitive orthogonal idempotents} is a set $\left\{\cb_1,\ldots,\cb_r\right\}$ such that
\begin{align*}
\sum_{i=1}^r \cb_i=\ee\quad\mbox{and}\quad\cb_i\cb_j=\delta_{ij}\cb_i\quad\mbox{for } 1\leq i\leq j\leq r.
\end{align*}
The size $r$ of such system is a constant called the \emph{rank} of $\En$. Any element $\xx$ of a Euclidean simple Jordan algebra can be written as $\xx=\sum_{i=1}^r\lambda_i\cb_i$ for some complete $\left\{\cb_1,\ldots,\cb_r\right\}$ system of primitive orthogonal idempotents. The real numbers $\lambda_i$, $i=1,\ldots,r$ are the \emph{eigenvalues} of $\xx$. One can then define \emph{determinant} of $\xx$ by $\det\xx=\prod_{i=1}^r\lambda_i$. 

If $\cb$ is a primitive idempotent of $\En$, the only possible eigenvalues of $\LL(\cb)$ are $0$, $\tfrac{1}{2}$ and $1$. We denote by $\En(\cb,0)$, $\En(\cb,\tfrac{1}{2})$ and $\En(\cb,1)$ the corresponding eigenspaces. The decomposition 
\begin{align*}
\En=\En(\cb,0)\oplus\En(\cb,\tfrac{1}{2})\oplus\En(\cb,1)
\end{align*}
is called the \emph{Peirce decomposition of $\En$ with respect to $\cb$}. Note that $\PP(\cb)$ is the orthogonal projection of $\En$ onto $\En(\cb,1)$.

Fix a complete system of orthogonal idempotents $\left(\cb_i\right)_{i=1}^r$. Then for any $i,j\in\left\{1,\ldots,r\right\}$ we write
\begin{align*}
\begin{split}
\En_{ii} & =\En(\cb_i,1)=\RR \cb_i, \\
\En_{ij} & = \En\left(\cb_i,\frac{1}{2}\right) \cap \En\left(\cb_j,\frac{1}{2}\right) \mbox{ if }i\neq j.
\end{split}
\end{align*}
It can be proved (see \cite[Theorem IV.2.1]{FaKo1994}) that
\begin{align*}
\En=\bigoplus_{i\leq j}\En_{ij}
\end{align*}
and
\begin{align*}
\begin{split}
\En_{ij}\cdot\En_{ij} & \subset\En_{ii}+\En_{jj}, \\
\En_{ij}\cdot\En_{jk} & \subset\En_{ik},\mbox{ if }i\neq k, \\
\En_{ij}\cdot\En_{kl} & =\{0\},\mbox{ if }\{i,j\}\cap\{k,l\}=\emptyset.
\end{split}\end{align*}
The dimension of $\En_{ij}$ is, for any $i\neq j$, a constant $d$ called the Peirce constant. When $\En$ is $S_r(\KK)$, if $\{e_1,\ldots,e_r\}$ is an orthonormal basis of $\RR^r$, then $\En_{ii}=\RR e_i e_i^\top$ and $\En_{ij}=\KK\left(e_i e_j^\top+e_je_i^\top\right)$ for $i<j$ and $d$ is equal to $dim_{|\RR}\KK$.
	
For $1\leq k \leq r$ let $P_k$ be the orthogonal projection onto $\En^{(k)}=\En(\cb_1+\ldots+\cb_k,1)$, $\det^{(k)}$ the determinant in the subalgebra $\En^{(k)}$, and, for $\xx\in\VV$, $\Delta_k(\xx)=\det^{(k)}(P_k(\xx))$. Then $\Delta_k$ is called the principal minor of order $k$ with respect to the Jordan frame $\{\cb_k\}_{k=1}^r$. Note that $\Delta_r(\xx)=\det\xx$. For $s=(s_1,\ldots,s_r)\in\RR^r$ and $\xx\in\VV$, we write
\begin{align*}
\Delta_s(\xx)=\Delta_1(\xx)^{s_1-s_2}\Delta_2(\xx)^{s_2-s_3}\ldots\Delta_r(\xx)^{s_r}.
\end{align*}
$\Delta_s$ is called a generalized power function. If $\xx=\sum_{i=1}^r\alpha_i\cb_i$, then $\Delta_s(\xx)=\alpha_1^{s_1}\alpha_2^{s_2}\ldots\alpha_r^{s_r}$. For $s\in\RR^r$ and $\lambda\in\RR$ we will write
\begin{align*}
s+\lambda=(s_1+\lambda,\ldots,s_r+\lambda).
\end{align*}

We will now introduce some basic facts about triangular group. For $\xx$ and $\yy$ in $\VV$, let $\xx\Box\yy$ denote the endomorphism of $\En$ defined by
\begin{align*}
\xx\Box\yy=\LL(\xx\yy)+\LL(\xx)\LL(\yy)-\LL(\yy)\LL(\xx).
\end{align*}
If $\cb$ is an idempotent and $\zz\in\En(\cb,\frac{1}{2})$ we define the \emph{Frobenius transformation $\tau_\cb(\zz)$ in $G$} by
\begin{align*}
\tau_\cb(\zz)=\exp(2\zz\Box\cb).
\end{align*}
Given a Jordan frame $\{\cb_i\}_{i=1}^r$, the subgroup of $G$, 
\begin{align*}
\TT=\left\{\tau_{\cb_1}(\zz^{(1)})\ldots\tau_{\cb_{r-1}}(\zz^{(r-1)})\PP\left(\sum_{i=1}^r \alpha_i\cb_i\right)\colon \alpha_i>0, \zz^{(j)}\in \bigoplus_{k=j+1}^r\En_{jk}\right\}
\end{align*}
is called the \emph{triangular group corresponding to the Jordan frame $\{\cb_i\}_{i=1}^r$}. For any $\xx$ in $\VV$ there exists a unique $t_{\xx}$ in $\TT$ such that $\xx=t_{\xx}\ee$, that is, there exist (see \cite[Theorem VI.3.5]{FaKo1994}) elements $\zz^{(j)}\in \bigoplus_{k=j+1}^r \En_{jk}$, $1\leq j\leq r-1$ and positive numbers $\alpha_1, \ldots ,\alpha_r$ such that
\begin{align*}
\xx=\tau_{\cb_1}(\zz^{(1)})\ldots\tau_{\cb_{r-1}}(\zz^{(r-1)})\left(\sum_{k=1}^r \alpha_k \cb_k \right).
\end{align*}
Mapping $\ww_2\colon\VV\to\TT, \xx\mapsto \ww_2(\xx)=t_{\xx}$ is the second important example of a multiplication algorithm.

For $\En=S_r(\RR)$ we have $\VV=\VV_+$. Let us define for $1\leq i,j\leq r$ matrix $\mu_{ij}=\left(\gamma_{kl}\right)_{1\leq k,l\leq r}$ such that $\gamma_{ij}=1$ and all other entries are equal $0$.
Then for Jordan frame $\left\{\cb_i\right\}_{i=1}^r$, where $\cb_k=\mu_{kk}$, $k=1,\ldots,r$, we have $\zz_{jk}=(\mu_{jk}+\mu_{kj})\in\En_{jk}$ and $\norm{\zz_{jk}}^2=2$, $1\leq j,k\leq r$, $j\neq k$.
If $\zz^{(i)}\in\bigoplus_{j=i+1}^r \En_{ij}$, $i=1,\ldots,r-1$, then there exists $\alpha^{(i)}=(\alpha_{i+1},\ldots,\alpha_r)\in\RR^{r-i}$ such that $\zz^{(i)}=\sum_{j=i+1}^r \alpha_j\zz_{ij}$. Then the Frobenius transformation reads below
$$\tau_{\cb_i}(\zz^{(i)})\xx=\mathcal{F}_i(\alpha^{(i)})\cdot\xx\cdot \mathcal{F}_i(\alpha^{(i)})^\top,$$
where $\mathcal{F}_i(\alpha^{(i)})$ is so called Frobenius matrix:
\begin{align*}
\mathcal{F}_i(\alpha^{(i)})=I+\sum_{j=i+1}^r \alpha_j \mu_{ji},
\end{align*}
ie. bellow $i$th one of identity matrix there is a vector $\alpha^{(i)}$, particularly
\begin{align*}
\mathcal{F}_2(\alpha^{(2)})=\begin{pmatrix}
  1    &   0    &   0    & \cdots & 0 \\
  0    &   1    &   0    & \cdots & 0 \\
  0    & \alpha_{3} &   1    & \cdots & 0 \\
\vdots & \vdots & \vdots & \ddots & \vdots \\
  0    & \alpha_{r} &   0    & \cdots & 1
\end{pmatrix}.
\end{align*}

It can be shown (\cite[Proposition VI.3.10]{FaKo1994}) that for each $t\in\TT$, $\xx\in\VV$ and $s\in\RR^r$,
\begin{align}\label{dett}
\Delta_s(t\xx)=\Delta_s(t\ee)\Delta_s(\xx).
\end{align}
This property actually characterizes function $\Delta_s$ - see Theorem \ref{dettth}.

\section{Probability distributions}
The beta-Riesz distribution on symmetric cones with parameters $(s,t)\in \RR^r \times \RR^r$ for $s_i>(i-1)d/2$, $t_i>(i-1)d/2$, $i=1,\ldots,r$, ($d$ is the Peirce constant) is defined by its density
\begin{align*}
\mathcal{BR}(s,t)(\mathrm{d}\xx)=\frac{1}{\mathcal{B}_\VV(s,t)}\Delta_{s-\dim\VV/r}(\xx)\Delta_{t-\dim\VV/r}(\ee-\xx)I_\DD(\xx)\,\mathrm{d}\xx,\quad\xx\in\VV,
\end{align*}
where $\DD=\left\{\xx\in\VV\colon\ee-\xx\in\VV\right\}$ is an analogue of $(0,1)$ interval on real line and 
\begin{align*}
\mathcal{B}_\VV(s,t)=\frac{\Gamma_\VV(s)\Gamma_\VV(t)}{\Gamma_\VV(s+t)}
\end{align*}
for gamma function of symmetric cone $\Gamma_\VV(s)=(2\pi)^{(\dim\VV-r)/2}\prod_{j=1}^r \Gamma(s_j-(j-1)\tfrac{d}{2})$ (see \cite[VII.1.1.]{FaKo1994}).

Beta distribution on symmetric cone $\VV$ is a special case of beta-Riesz distribution for $s_1=\ldots=s_r=p>\dim\VV/r-1$ and $t_1=\ldots=t_r=q>\dim\VV/r-1$ with density
\begin{align*}
\mathcal{B}(p,q)(\mathrm{d}\xx)=\frac{1}{\mathcal{B}_\VV(p,q)}(\det\xx)^{p-\dim\VV/r}\det(\ee-\xx)^{q-\dim\VV/r}I_\DD(\xx)\,\mathrm{d}\xx,\quad\xx\in\VV,
\end{align*}
where $\mathcal{B}_\VV(p,q)=\frac{\Gamma_\VV(p)\Gamma_\VV(q)}{\Gamma_\VV(p+q)}$ and $\Gamma_\VV(p):=\Gamma_\VV(p,\ldots,p)$.
Basic properties of beta and beta-Riesz distributions on $\VV_+$ are given in \cite{Ha2005,Zine2012} and of beta distribution on $\VV_+$ in \cite{OlRu1964}.
For some recent advances in extending beta distribution the reader is referred to \cite{NRG2013}.

\section{Characterization of generalized beta distribution}\label{beta}
Henceforth we will denote by $\VV$ an irreducible symmetric cone of rank $r$. The densities of generalized beta distributions will be given in terms of \emph{$\ww$-logarithmic} functions, that is functions $f\colon\VV\to\RR$ that satisfies the following functional equation 
\begin{align}\label{wC}
f(\xx)+f(\ww(\ee)\yy)=f(\ww(\xx)\yy),\quad (\xx,\yy)\in\VV^2,
\end{align}
where $\ww$ is a multiplication algorithm. If $f$ is $\ww$-logarithmic, then $e^f$ is said to be \emph{$\ww$-multiplicative}. Functional equation \eqref{wC} for $\ww_1(\xx)=\PP(\xx^{1/2})$ on $\VV_+$ was already considered in \cite{BW2003} for differentiable functions and in \cite{Molnar2006} for continuous functions on real or complex Hermitian positive definite matrices of rank greater than $2$. Without any regularity assumptions it was solved on the Lorentz cone by \cite{Wes2007L}. The general forms of $\ww_1$- and $\ww_2-$logarithmic functions without any regularity assumptions were given in \cite{wC2013}.

It should be stressed that there exists infinite number of multiplication algorithms. If $\ww$ is a multiplication algorithm, then trivial extensions are given by $\ww^{(k)}(\xx) = \ww(\xx)k$, where $k\in K$ is fixed and $K$ is defined by \eqref{defk}. One may consider also multiplication algorithms of the form $P(\xx^{\alpha})t_{\xx^{1-2\alpha}}$, $\alpha\in\RR$, which interpolates between the two main examples: $\ww_1$ (which is $\alpha= 1/2$) and $\ww_2$ (which is $\alpha = 0$). In general, any multiplication algorithm may be written in the form $\ww(\xx) = \PP(\xx^{1/2})k_\xx$, where $k_\xx\in K$.

To define the transformation of random variables we will use two multiplication algorithms, $\ww$ and $\widetilde{\ww}$. Let $\gw$ and $\widetilde{\gw}$ be the corresponding division algorithms, that is, $\gw=\ww^{-1}$ and $\widetilde{\gw}=\widetilde{\ww}^{-1}$. Henceforth we will assume that $\ww$ additionally satisfies the following natural conditions
\begin{enumerate}[A.]
\item $\ww$ is homogeneous of degree $1$, that is $\ww(s\xx)=s\ww(\xx)$ for any $s>0$ and $\xx\in\VV$,
\item continuity in $\ee$, that is $\lim_{\xx\to\ee}\ww(\xx)=\ww(\ee)$,
\item surjectivity of the mapping $\VV\ni\xx\mapsto\gw(\xx)\ee\in\VV$,
\item differentiability of the mapping $\VV\ni\xx\mapsto\ww(\xx)$,
\end{enumerate}
and the same is assumed for $\widetilde{\ww}$.

Conditions $A-C$ are assumed in order to use the result of \cite{BKInf2014} regarding the generalized fundamental equation of information on $\VV$ (see Theorem \ref{inform} below) and $D$ is assumed to ensure that the Jacobian of the considered transformation exists. By $\we$ and $\widetilde{\ww}_\ee$ we will denote $\ww(\ee)$ and $\widetilde{\ww}(\ee)$ respectively. 

We start with the direct result, where we show that if $X$ and $Y$ have densities of the form \eqref{wbeta}, then the transformed variables are independent. Recall that $\VV$ is an irreducible symmetric cone of rank $r$ and $\ww$-multiplicative function $f$ satisfies the following functional equation
$$f(\xx)f(\ww(\ee)\yy)=f(\ww(\xx)\yy)$$
for any $(\xx,\yy)\in\VV$.

\begin{theorem}\label{direct}
Assume that multiplication algorithms $\ww$ and $\widetilde{\ww}$ are differentiable (condition $D$) and let $X$ and $Y$ be independent random variables valued in $\DD$ with densities of the form
\begin{align}\begin{split}\label{wbeta}
f_X(\xx) &= c_X (\det\xx)^{\dim\VV/r}i(\xx)h(\xx)f(\ee-\xx)I_\DD(\xx), \\
f_Y(\xx) &= c_Y h(\we\xx)i(\ee-\we\xx)I_\DD(\xx),
\end{split}\end{align}
where 
\begin{itemize}
\item $i$ is $\ww$- and $\widetilde{\ww}$-multiplicative, 
\item $f$ is $\widetilde{\ww}$ multiplicative,
\item $h$ is $\ww$-multiplicative,
\end{itemize}
and $c_X$, $c_Y$ are normalizing constants.
Then 
\begin{align*}
U=\ee-\ww(X)Y\quad\mbox{ and }\quad V=\widetilde{\gw}(U)(\ee-X)
\end{align*}
are independent random variables.
\end{theorem}
\begin{proof}
Define the mapping $\psi\colon\DD^2\to\DD^2$ by formula
\begin{align*}
\psi(\xx,\yy)=\left(\ee-\ww(\xx)\yy,\widetilde{\gw}(\ee-\ww(\xx)\yy)(\ee-\xx)\right),
\end{align*} 
where $\widetilde{\gw}=\widetilde{\ww}^{-1}$.
Then we have $(U,V)=\psi(X,Y)$ and the inverse mapping $\psi^{-1}\colon\DD^2\to\DD^2$ is given by
\begin{align*}
(\xx,\yy)=\psi^{-1}(\ub,\vb)=\left(\ee-\widetilde{\ww}(\ub)\vb,\gw(\ee-\widetilde{\ww}(\ub)\vb)(\ee-\ub)\right),
\end{align*}
where $\gw=\ww^{-1}$. Hence $\psi$ is a bijection. We will find the Jacobian of $\psi^{-1}$ in two steps. Let us observe that $\psi^{-1}=\phi_2\circ\phi_1$ with
\begin{align*}
\phi_1(\ub,\vb) & =(\ee-\widetilde{\ww}(\ub)\vb,\ee-\ub)=(\ab,\bb), \\
\phi_2(\ab,\bb) & =(\ab ,\gw(\ab)\bb)=(\xx,\yy).
\end{align*}
Denote by $J_i$ the Jacobian of mapping $\phi_i$, $i=1,2$.
We have
\begin{align*}
J_1= & 
\left|
\begin{array}{cc}
\mathrm{d}\ab/\mathrm{d}\ub & \mathrm{d}\ab/\mathrm{d}\vb \\
\mathrm{d}\bb/\mathrm{d}\ub & \mathrm{d}\bb/\mathrm{d}\vb
\end{array}
\right|
=
\left| 
\begin{array}{cc}
\mathrm{d}\ab/\mathrm{d}\ub & -\widetilde{\ww}(\ub) \\
-\mathrm{Id}_\VV & 0
\end{array}
\right| 
=
\DDet(\widetilde{\ww}(\ub)) \\
\intertext{and}
J_2= & 
\left|
\begin{array}{cc}
\mathrm{d}\xx/\mathrm{d}\ab & \mathrm{d}\xx/\mathrm{d}\bb \\
\mathrm{d}\yy/\mathrm{d}\ab & \mathrm{d}\yy/\mathrm{d}\bb
\end{array}
\right|
=
\left| 
\begin{array}{cc}
\mathrm{Id}_\VV & 0 \\
\mathrm{d}\yy/\mathrm{d}\ab & \gw(\ab)
\end{array}
\right| 
=
\DDet(\gw(\ab)).
\intertext{Finally, by \eqref{zzz}, we get}
J & =J_1J_2=\left(\frac{\det\ub}{\det(\ee-\widetilde{\ww}(\ub)\vb)} \right)^{\tfrac{\dim\VV}{r}}.
\end{align*}
The joint density $f_{(U,V)}$ of $(U,V)$ is given by
\begin{align}\label{fuv}\begin{split}
f_{(U,V)}(\ub,\vb)=f_X(\ee-\widetilde{\ww}(\ub)\vb)f_Y(\gw(\ee-\widetilde{\ww}(\ub)\vb)(\ee-\ub))\left(\frac{\det\ub}{\det(\ee-\widetilde{\ww}(\ub)\vb)} \right)^{\tfrac{\dim\VV}{r}},
\end{split}\end{align}
where $f_X$ and $f_Y$ denote the densities of $X$ and $Y$, respectively. 
Inserting \eqref{wbeta} into \eqref{fuv} and repeatedly using multiplicative properties of respective functions (that is, if $h$ is $\ww$-multiplicative, then $h(\xx)h(\we g(\xx)\yy)=h(\yy)$ for any $\xx,\yy\in\VV$), we obtain
\begin{align*}
f_{(U,V)}(\ub,\vb)=c_Xc_Y(\det\ub)^{\dim\VV/r}i(\ub)f(\ub)h(\ee-\ub)I_\DD(\ub)\cdot f(\widetilde{w}_\ee\vb)i(\ee-\widetilde{w}_\ee\vb)I_\DD(\vb),
\end{align*}
what completes the proof.
\end{proof}
\begin{remark}\label{remX}
Note that if $i(\xx)=(\det\xx)^{p_1-\dim\VV/r}$, $f(\xx)=(\det\xx)^{p_2-\dim\VV/r}$ and $h(\xx)=(\det\xx)^{p_3-\dim\VV/r}$ with $p_i>\dim\VV/r-1$, $i=1,2,3$, then $(X,Y)\sim\mathcal{B}(p_1+p_3,p_2)\otimes\mathcal{B}(p_3,p_1)$ and $(U,V)\sim\mathcal{B}(p_1+p_2,p_3)\otimes\mathcal{B}(p_2,p_1)$, regardless of the choice of $\ww$ and $\widetilde{\ww}$.
\end{remark}

In order to prove the harder part of the characterization we will need the following result regarding the solution to fundamental equation of information on symmetric cones (see  \cite[Theorem 3.5]{BKInf2014}).
Recall that $\DD=\left\{\xx\in\VV\colon\ee-\xx\in\VV\right\}$ and define 
\begin{align*}
\DD_0=\{(\ab,\bb)\in\DD^2\colon \ab+\bb\in\DD\}.
\end{align*}
\begin{theorem}\label{inform}
Let $a, b, c, d\colon\DD\to\RR$ be continuous functions that satisfy the following functional equation
\begin{align*}
a(\xx)+b(\gw(\ee-\xx)\yy)=c(\yy)+d(\widetilde{\gw}(\ee-\yy)\xx),\quad (\xx,\yy)\in\DD_0.
\end{align*}
If multiplication algorithms $\ww=\gw^{-1}$ and $\widetilde{\ww}=\widetilde{\gw}^{-1}$ satisfy conditions $A-C$, then there exist real constants $C_i$, $i=1,\ldots,4$, and continuous functions $h_i$, $i=1,2,3$, where 
\begin{itemize}
\item $h_1$ is $\ww$- and $\widetilde{\ww}$-logarithmic, 
\item $h_2$ is $\widetilde{\ww}$ logarithmic, 
\item $h_3$ is $\ww$-logarithmic,
\end{itemize}
such that for any $\xx\in\DD$,
\begin{align*}
a(\xx) & = h_1(\ee-\xx)+h_2(\xx)+h_3(\ee-\xx)+C_1, \\
b(\xx) & = h_1(\ee-\we\xx)+h_3(\we\xx)+C_2, \\
c(\xx) & = h_1(\ee-\xx)+h_2(\ee-\xx)+h_3(\xx)+C_3, \\
d(\xx) & = h_1(\ee-\widetilde{\ww}_\ee\xx)+h_2(\widetilde{\ww}_\ee\xx)+C_4,
\end{align*}
and $C_1+C_2 = C_3+C_4$.
\end{theorem}

We are now ready to prove the main theorem.
\begin{theorem}[Characterization of generalized beta distributions]\label{betath}
Let $X$ and $Y$ be independent random variables valued in $\DD$ with continuous and strictly positive densities.  
Let additionally $\psi\colon\DD^2\to\DD^2$ be a mapping defined through 
\begin{align*}
\psi(\xx,\yy)=\left(\ee-\ww(\xx)\yy,\widetilde{\gw}(\ee-\ww(\xx)\yy)(\ee-\xx)\right),
\end{align*} 
where $\ww=\gw^{-1}$ and $\widetilde{\ww}=\widetilde{\gw}^{-1}$ are multiplication algorithms satisfying conditions $A-D$. If components of vector $(U,V)=\psi(X,Y)$ are independent, then there exist continuous functions $i$, $f$, $g$, where
\begin{itemize}
\item $i$ is $\ww$- and $\widetilde{\ww}$-multiplicative, 
\item $f$ is $\widetilde{\ww}$ multiplicative, 
\item $h$ is $\ww$-multiplicative,
\end{itemize}
and \eqref{wbeta} holds.
\end{theorem}
\begin{proof}
Let us note that, as in the proof of Theorem \ref{direct}, the joint density of $(U,V)$ has the form \eqref{fuv}. This equality is satisfied almost everywhere with respect to Lebeasgue measure. According to the assumption that $(U,V)$ has independent components, we have $f_{(U,V)}(\ub,\vb)=f_U(\ub)f_V(\vb)$.

Since the respective densities are continuous, \eqref{fuv} holds true for any $\ub,\vb\in\DD$. Taking logarithm of both sides of \eqref{fuv} (it is permitted, since densities are strictly positive on $\DD$), we obtain
\begin{align*}
a\left(\widetilde{\ww}(\ub)\vb\right)+b\left(\gw(\ee-\widetilde{\ww}(\ub)\vb)(\ee-\ub)\right)=c(\ee-\ub)+d(\vb),\quad (\ub,\vb)\in\DD^2,
\end{align*}
where
\begin{align*}
a(\ub) & = \log\,f_X(\ee-\ub)-\frac{\dim\VV}{r}\log\det(\ee-\ub), \\
b(\ub) & = \log\,f_Y(\ub), \\
c(\ub) & = \log\,f_U(\ee-\ub)-\frac{\dim\VV}{r}\log\det(\ee-\ub),\\
d(\ub) & = \log\,f_V(\ub),
\end{align*}
for $\ub\in\DD$. Let us take $\ub=\ee-\yy$ and $\vb=\widetilde{\gw}(\ee-\yy)\xx$. Then $\yy\in\DD$, because $\ub\in\DD$. Moreover, it is clear that $\xx\in\VV$. Since $\vb\in\DD$ and $\ee-\vb=\widetilde{\gw}(\ee-\yy)(\ee-\xx-\yy)\in\DD$, we have $\ee-(\xx+\yy)\in\VV$. Thus,
\begin{align}\label{A}
a(\xx)+b(\gw(\ee-\xx)\yy)=c(\yy)+d(\widetilde{\gw}(\ee-\yy)\xx)
\end{align}
for any $(\xx,\yy)\in\DD_0=\{(\ab,\bb)\in\DD^2\colon \ab+\bb\in\DD\}$. Theorem \ref{inform} implies that there exist continuous functions $h_1$, $h_2$ and $h_3$ such that
\begin{itemize}
\item $h_1$ is $\ww$- and $\widetilde{\ww}$-logarithmic function, 
\item $h_2$ is $\widetilde{\ww}$ logarithmic, 
\item $h_3$ is $\ww$-logarithmic,
\end{itemize}
and
\begin{align*}
a(\xx) & = h_1(\ee-\xx)+h_2(\xx)+h_3(\ee-\xx)+C_1, \\
b(\xx) & = h_1(\ee-\we\xx)+h_3(\we\xx)+C_2,
\end{align*}
for real constants $C_i$, $i=1,2$. That is, for $\xx\in\DD$ we have
\begin{align}\label{hi}
f_X(\xx) & =e^{a(\ee-\xx)+\frac{\dim\VV}{r}\log\det\xx}=e^{C_1}\det(\xx)^{\frac{\dim\VV}{r}} e^{h_1(\xx)} e^{h_3(\xx)} e^{h_2(\ee-\xx)}, \\
f_Y(\xx) & =e^{b(\xx)}=e^{C_2} e^{h_3(\we\xx)} e^{h_1(\ee-\we\xx)},
\end{align}
what is essentially \eqref{wbeta} for $i(\xx)=e^{h_1(\xx)}$, $f(\xx)=e^{h_2(\xx)}$ and $h(\xx)=e^{h_3(\xx)}$.
\end{proof}

As was mentioned earlier, the general form of $\ww$-logarithmic functions is known in two basic examples, namely $\ww=\ww_1=\PP(\xx^{1/2})$ and $\ww=\ww_2=t_\xx\in\TT$. These forms (see Theorem \ref{detwth} and \ref{dettth} below) will be needed in the proof of the main theorem. The proofs of these results may be found in \cite{wC2013}. Function $H$ is called generalized logarithmic, if $H(ab)=H(a)+H(b)$ for any positive $a$ and $b$.

\begin{theorem}[$\ww_1$-logarithmic Cauchy functional equation]\label{detwth}
Let $f\colon \VV\to\RR$ be a function such that
\begin{align*}
f(\xx)+f(\yy)=f\left(\PP\left(\xx^{1/2}\right)\yy\right),\quad (\xx,\yy)\in\VV^2.
\end{align*}
Then there exists a generalized logarithmic function $H$ such that for any $\xx\in\VV$,
\begin{align*}
f(\xx)=H(\det\xx).
\end{align*}
\end{theorem}

\begin{theorem}[$\ww_2$-logarithmic Cauchy functional equation]\label{dettth}
Let $f\colon \VV\to \RR$ be a function satisfying
\begin{align*}
f(\xx)+f(\yy)=f(t_{\yy}\xx)
\end{align*}
for any $\xx$ and $\yy$ in the cone $\VV$ of rank $r$, $t_{\yy}\in\TT$, where $\TT$ is the triangular group with respect to the Jordan frame $\left\{\cb_i\right\}_{i=1}^r$. Then there exist generalized logarithmic functions $H_1,\ldots, H_r$ such that for any $\xx\in\VV$,
\begin{align*}
f(\xx)=\sum_{k=1}^r H_k(\Delta_k(\xx)),
\end{align*}
where $\Delta_k$ is the principal minor of order $k$ with respect to $\left\{\cb_i\right\}_{i=1}^r$.
\end{theorem}

\begin{remark}\label{dettrem}
If we impose on $f$ in Theorem \ref{dettth} some mild conditions (eg. measurability), then there exists $s\in\RR^r$ such that for any $\xx\in\VV$,
$$f(\xx)=\log\Delta_s(\xx).$$
\end{remark}

We may now give the specification of Theorem \ref{betath}, when \eqref{wbeta} is known explicitly. 
For every generalized multiplication $\ww$ and $\widetilde{\ww}$, the family of generalized beta measures (as defined in \eqref{wbeta}) contains the beta laws (see Remark \ref{remX}). For $\ww=\widetilde{\ww} = \ww_1$, there are no other distributions, while for $\ww=\widetilde{\ww}=\ww_2$ generalized beta measures consist of the beta-Riesz distributions. It is an open question whether there is a generalized multiplication $\ww$ that leads to other probability measures in this family.

Define $\underline{1}:=(1,\ldots,1)\in\RR^r$. Recall that $\VV$ is an irreducible symmetric cone of rank $r$, $d$ is its Peirce constant and $\DD=\left\{\xx\in\VV\colon\ee-\xx\in\VV\right\}$.
\begin{theorem}[Characterization of beta and beta-Riesz distributions]\label{betathS}
Let $X$ and $Y$ be independent random variables valued in $\DD$ with continuous and strictly positive densities.  
Let additionally $\psi\colon\DD^2\to\DD^2$ be a mapping defined through 
\begin{align*}
\psi(\xx,\yy)=\left(\ee-\ww(\xx)\yy,\widetilde{\gw}(\ee-\ww(\xx)\yy)(\ee-\xx)\right),
\end{align*} 
where $\ww=\gw^{-1}$ and $\widetilde{\ww}=\widetilde{\gw}^{-1}$ are multiplication algorithms satisfying conditions $A-D$. Assume that components of vector $(U,V)=\psi(X,Y)$ are independent.

If 
\begin{enumerate}
\item $\ww(\xx)=\widetilde{\ww}(\xx)=\PP(\xx^{1/2})$, then there exist constants $p_i>\dim\VV/r-1$, $i=1,2,3$, such that 
\begin{align*}
X\sim \mathcal{B}(p_1+p_3,p_2)\quad\mbox{ and }\quad Y\sim \mathcal{B}(p_3,p_1),
\end{align*}
\item $\ww(\xx)=\widetilde{\ww}(\xx)=t_\xx$, then there exist vectors $s_i=(s_{i,j})_{j=1}^r$, $s_{i,j}>(j-1)d/2$, $i=1,2,3$, $j=1,\ldots,r$, such that 
\begin{align*}
X\sim \mathcal{BR}(s_1+s_3,s_2)\quad\mbox{ and }\quad Y\sim \mathcal{BR}(s_3,s_1),
\end{align*}
\item $\ww(\xx)=\PP(\xx^{1/2})$ and $\widetilde{\ww}(\xx)=t_\xx$ , then there exist constants $p_i>\dim\VV/r-1$, $i=1,3$ and vector $s_2=(s_{2,j})_{j=1}^r$, $s_{2,j}>(j-1)d/2$, such that 
\begin{align*}
X\sim \mathcal{BR}((p_1+p_3)\underline{1},s_2)\quad\mbox{ and }\quad Y\sim\mathcal{B}(p_3,p_1),
\end{align*}
\item $\ww(\xx)=t_\xx$ and $\widetilde{\ww}(\xx)=\PP(\xx^{1/2})$, then there exist constants $p_i>\dim\VV/r-1$, $i=1,2$ and vector $s_3=(s_{3,j})_{j=1}^r$, $s_{3,j}>(j-1)d/2$, such that 
\begin{align*}
X\sim \mathcal{BR}(p_1\underline{1}+s_3,p_2\underline{1})\quad\mbox{ and }\quad Y\sim\mathcal{BR}(s_3,p_1\underline{1}).
\end{align*}
\end{enumerate}
\end{theorem}
\begin{proof}
We start with \eqref{hi}. 
If $\ww(\xx)=\ww_1(\xx)=\PP(\xx^{1/2})$, then by Theorem \ref{detwth} we know that there exist constants $\kappa_i\in\RR$ such that $h_i(\xx)=\kappa_i\log\det\xx$, $i=1,2,3$. Thus $X$ follows $\mathcal{B}(p_1+p_3,p_2)$ distribution and $Y$ follows $\mathcal{B}(p_3,p_1)$ distribution, where $p_i=\kappa_i+\dim\VV/r>\dim\VV/r-1$, $i=1, 2,3$.

If, in turn, $\ww(\xx)=\ww_2(\xx)=t_\xx$, then by Theorem \ref{dettth} and Remark \ref{dettrem} we get the existence of vectors $t_i\in\RR^r$ such that $h_i(\xx)=\log\Delta_{t_i}(\xx)$, $i=1,2,3$. So $X$ follows $\mathcal{BR}(s_1+s_3,s_2)$ distribution and $Y$ follows $\mathcal{BR}(s_3,s_1)$ distribution, where $s_i=t_i+\dim\VV/r$, $i=1,2,3$ are such that $s_{i,j}>(j-1)d/2$, $i=1,2,3$, $j=1,\ldots,r$.

Points $(3)$ and $(4)$ are proved analogously.
\end{proof}

\section{Distributions invariant under the group of automorphisms}
In the famous Lukacs-Olkin-Rubin Theorem (see \cite{OlRu1964} for $\VV_+$ case, \cite{CaLe1996} for all irreducible symmetric cones and \cite[Remark 4.4]{BKLOR2014} for its density version), the following independence property was analyzed: assume $X$ and $Y$ are independent random variables valued in $\VV$ and $V=X+Y$ and $U=\gw(X+Y)X$ (here $\gw=\ww^{-1}$) are also independent (supplemented with some technical assumptions). If the distribution of $U$ is invariant under the group $K$ of automorphisms, that is $kU\stackrel{d}{=}U$ for any $k\in K$, then $X$ and $Y$ follow Wishart distribution with the same scale parameter, regardless of the choice of multiplication algorithm $\ww=\gw^{-1}$. In that case $U$ was beta distributed for any measurable division algorithm $\gw$. Similar approach in our case also leads to the characterization of beta distribution on $\VV$ (see Theorem \ref{betathK} below).

Function $f\colon\VV\to\RR$ is called $K$-invariant if $f(k\xx)=f(\xx)$ for any $k\in K$ and $\xx\in\VV$. We will need the following result of \cite{BKInf2014}, where compared to Theorem \ref{inform}, additional assumption of $K$-invariance is imposed on unknown functions.
\begin{theorem}\label{cor2}
Let $a, b, c, d$, $\ww, \widetilde{\ww}$ be as in Theorem \ref{inform}, but assume additionally that any two unknown functions are $K$-invariant. Then, there exist constants $\kappa_j$, $j=1,2,3$ and $C_i$, $i=1,\ldots,4$, such that for any $\xx\in\DD$,
\begin{align}\begin{split}\label{tryk}
a(\xx) & = (\kappa_1+\kappa_3)\log\det(\ee-\xx)+\kappa_2\log\det\xx+C_1, \\
b(\xx) & = \kappa_1\log\det(\ee-\xx)+\kappa_3\log\det\xx+C_2, \\
c(\xx) & = (\kappa_1+\kappa_2)\log\det(\ee-\xx)+\kappa_3\log\det\xx+C_3, \\
d(\xx) & = \kappa_1\log\det(\ee-\xx)+\kappa_2\log\det\xx+C_4,
\end{split}\end{align}
and $C_1+C_2 = C_3+C_4$.
\end{theorem}

\begin{theorem}[Characterization of beta distribution]\label{betathK}
Let $X$ and $Y$ be independent random variables valued in $\DD$ with continuous and strictly positive densities. Assume additionally that the distributions of $X$ and $Y$ are invariant under the group $K$ of automorphisms. Let $\psi\colon\DD^2\to\DD^2$ be a mapping defined through 
\begin{align*}
\psi(\xx,\yy)=\left(\ee-\ww(\xx)\yy,\widetilde{\gw}(\ee-\ww(\xx)\yy)(\ee-\xx)\right),
\end{align*} 
where $\ww=\gw^{-1}$ and $\widetilde{\ww}=\widetilde{\gw}^{-1}$ are multiplication algorithms satisfying conditions $A-D$. If components of vector $(U,V)=\psi(X,Y)$ are independent, then there exist constants $p_i>\dim\VV/r-1$, $i=1,2,3$, such that $X\sim \mathcal{B}(p_1+p_3,p_2)$ and $Y\sim \mathcal{B}(p_3,p_1)$.
\end{theorem}
\begin{proof}
The proof begins exactly the same as in Theorem \ref{betath}; we start with \eqref{A}. If distributions of $X$ and $Y$ are invariant under the group of automorphisms, then their densities are $K$-invariant functions, that is $f_X(k\xx)=f_X(\xx)$ and $f_Y(k\xx)=f_Y(\xx)$ for any $k\in K$ and $\xx\in\DD$. From this we conclude that $a(\ub)=\log\,f_X(\ee-\ub)-\frac{\dim\VV}{r}\log\det(\ee-\ub)$ and $b(\ub)= \log\,f_Y(\ub)$ are also $K$-invariant, thus by Theorem \ref{cor2} we get the assertion.
\end{proof}

\subsection*{Acknowledgment} I am thankful to J. Weso{\l}owski for drawing my attention to this problem and for helpful discussions. This research was partially supported by NCN Grant No. 2012/05/B/ST1/00554. 

\bibliographystyle{plainurl}

\bibliography{Bibl}

\end{document}